\documentclass{amsart}
\usepackage[utf8]{inputenc}
\usepackage{amssymb,bm,amsfonts, stmaryrd, amsmath, amsthm, enumerate, mathtools}

\usepackage[colorlinks]{hyperref}
\hypersetup{
 colorlinks=true,
 linkcolor=blue,
 filecolor=magenta, 
 urlcolor=cyan,
}
\usepackage{wrapfig,tikz, tikz-cd}
\usetikzlibrary{arrows}

\usepackage[capitalise]{cleveref}
\crefformat{equation}{(#2#1#3)}
\crefrangeformat{equation}{(#3#1#4--#5#2#6)}
\crefformat{enumi}{(#2#1#3)}
\crefrangeformat{enumi}{(#3#1#4--#5#2#6)}

\newtheorem{theorem}{Theorem}[section]
\newtheorem{lemma}[theorem]{Lemma}

\newtheorem{corollary}[theorem]{Corollary}

\newcounter{intro}
\newtheorem{introthm}[intro]{Theorem}

\theoremstyle{definition}
\newtheorem{definition}[theorem]{Definition}

\newtheorem{remark}[theorem]{Remark}
\newtheorem{construction}[theorem]{Construction}
\newtheorem{notation}[theorem]{Notation}
\newtheorem{discussion}[theorem]{Discussion}
\newtheorem{chunk}[theorem]{}

\newtheorem*{ack}{Acknowledgements}

\usepackage{comment}

\newcommand{\ind}{{\operatorname{ind}}}
\newcommand{\gind}{{\operatorname{ind^\gamma}}}

\newcommand{\Tor}{{\operatorname{Tor}}}

\newcommand{\del}{\partial}

\DeclareMathOperator{\ch}{char}

\DeclareMathOperator{\comp}{\gamma}

\newcommand{\xra}{\xrightarrow}

\newcommand{\vp}{\varphi}

\newcommand{\shift}{{\scriptstyle\mathsf{\Sigma}}}

\newcommand{\lotimes}{\otimes^{\sf L}}

\newcommand{\m}{\mathfrak{m}}

\newcommand{\ve}{\varepsilon}

\newcommand{\gdev}{\varepsilon^\gamma}

\title[Dg resolutions in prime characteristic]{A comparison of dg algebra resolutions with prime residual characteristic}

\author{Michael DeBellevue}
\address{Mathematics Department, Syracuse University, Syracuse, NY 13244 U.S.A.}
\email{mpdebell@syr.edu}
\author{Josh Pollitz}
\address{Department of Mathematics,
University of Utah, Salt Lake City, UT 84112, U.S.A.}
\email{pollitz@math.utah.edu}

\date{\today}


\begin{document}

\begin{abstract}
In this article we fix a prime integer $p$ and compare certain dg algebra resolutions over a local ring whose residue field has characteristic $p$. Namely, we show that given a closed surjective map between such algebras there is a precise description for the minimal model in terms of the acyclic closure, and that the latter is a quotient of the former. A first application is that the homotopy Lie algebra of a closed surjective map is abelian. We also use these calculations to show deviations enjoy rigidity properties which detect the (quasi-)complete intersection property.
\end{abstract}
\maketitle
\section*{Introduction}\label{introduction}
This work concerns homological properties of a surjective map of commutative noetherian local rings $\vp\colon R\rightarrow S$. 
These properties are studied through certain differential graded (from now on abbreviated to dg) algebra resolutions of $S$ over $R$, which provide a relationship between homological and ring-theoretic data of $\vp$.  
A dg  algebra resolution consists of a free resolution of $S$ over $R$ further equipped with a graded-commutative multiplication compatible with the differential; see \cref{sec:dg} and the references therein for more information. 

The prototypical example of a dg algebra resolution is the Koszul complex resolving the residue field of a regular local ring.  
Its multiplicative structure is that of the exterior algebra, and its differential maps the exterior variables bijectively to a set of minimal generators of the maximal ideal of $R$. 
When $R$ is singular the Koszul complex is not acyclic, but Tate \cite{Tate:1957} realized that a dg algebra resolution may still be obtained by adjoining additional variables to the Koszul complex.  

Tate's construction consists of minimally adjoining exterior and divided power variables in odd and even degrees, respectively.  
Its importance was demonstrated by Schoeller~\cite{Schoeller:1967} and Gulliksen~\cite{Gulliksen:1968}, who proved that this always yields a minimal resolution of the residue field. 
More generally, Tate's construction always produces a (not necessarily minimal) dg algebra resolution of $S$ over $R$ called the \emph{acyclic closure} of $\vp$, denoted $R\langle Y\rangle$ with $Y$ the set of exterior and divided power variables.

If polynomial variables are used in lieu of divided power variables one obtains the minimal models which were imported to local algebra from rational homotopy theory by Avramov~\cite{Avramov:1984}; cf.\@ \cref{c:minmodel}.
We write $R[X]$ for a minimal model for $\vp$ where $X$ is the set of exterior and polynomial variables adjoined.  
Their use has been instrumental in the detection of the complete intersection property, and showing that the gap between complete intersection rings and other singularities is vast. 
In particular, they were applied to show the Betti numbers of the residue field have polynomial growth over complete intersection rings and exponential growth in all other cases~\cite{Avramov:1984,Avramov:1999}.
More recently, minimal models were also used by Briggs \cite{Briggs:2020} to settle Vasconcelos conjecture on the conormal module~\cite{Vasconcelos:1978}. 

When the common residue field $k$ of $R$ and $S$ has characteristic zero, it is well known that the acyclic closure and minimal model of $\vp\colon R\rightarrow S$ coincide (see \cref{r:minmodelequalacycliclcosure}); in general, they can differ drastically.  The main result in this article gives a precise relationship between these two resolutions when $k$ has characteristic $p>0$ and the acyclic closure of $\vp$ is minimal as a complex of $R$-modules; such maps are called closed and were introduced in \cite{Avramov/Iyengar:2003}.

\begin{introthm}\label{introthm1}
Let $\vp\colon R\rightarrow S$ be a surjective map of local rings whose residue field $k$ has characteristic $p>0$. 
If $\vp$ is closed, then its acyclic closure $R\langle Y\rangle$ is a quotient of its minimal model $R[X]$, and there is an exact sequence of graded $k$-spaces
\[
0\to \bigoplus_{i=1}^\infty ( kY_{\rm{even}}^{(p^i)}\oplus \shift kY_{\rm{even}}^{(p^i)} )\to kX\to kY\to 0\,,
\] where $\shift$ is the usual suspension functor of a graded object and $Y_{\rm{even}}^{(p^i)}$ consists of all divided power monomials $y^{(p^i)}$ for $y\in Y_{\rm{even}}$.
\end{introthm}

Closed maps arise frequently. For example, large and quasi-complete intersection homomorphisms are well-studied classes of ring maps, and each belong to the class of closed maps; see \cref{c:closedexamples} for more details.  Understanding the structure of quasi-complete intersection maps is of particular importance due to their role in a long-standing conjecture of Quillen~\cite{Quillen:1970}.  The conclusion of \cref{introthm1} holds under the more technical assumption that $\vp$ is \emph{weakly-closed}, as defined in \cref{d:weaklyclosed}; cf.\@ \cref{r:question}. A refinement of \cref{introthm1} can be found in \cref{t:refinement}.

The main application of \cref{t:refinement} is a computation of the homotopy Lie algebra $\pi(F^\vp)$ of $\vp$ whenever $\vp$ is closed; this calculation is purely in terms of the acyclic closure of $\vp$.
The homotopy Lie algebra---also adopted from rational homotopy theory to local algebra by Avramov \cite{Avramov:1984}---is a graded Lie algebra naturally associated to a local homomorphism. Properties of $\pi(F^\vp)$ relate to those of $\vp$, such as whether $\vp$ is complete intersection~\cite{Avramov/Halperin:1987} or Golod~\cite{Avramov:1986}, making $\pi(F^\vp)$ a useful computational tool. For example, it was recently used by Briggs to settle a long-standing conjecture of Vasconcelos~\cite{Briggs:2020}.  See \cref{c:HLA} for more details on the homotopy Lie algebra.

\begin{introthm}\label{introthm2}
If $\vp$ is closed, then $\pi(F^\vp)$ is an abelian Lie algebra with $k$-basis dual to 
\[
\begin{cases} \shift Y & \ch k=0\\
\shift Y\cup \bigcup_{i=1}^\infty \left(\shift Y_{\rm{even}}^{(p^i)}\cup \shift^{2} Y_{\rm{even}}^{(p^i)}\right) & \ch k >0\,.
\end{cases}
\]
 Furthermore if $\ch k\neq2$, then $\pi(F^\vp)$ also has trivial reduced square. 
\end{introthm}

By an abelian Lie algebra we mean a Lie algebra whose bracket is trivial. In characteristic zero, \cref{introthm2} follows directly from the equality of acyclic closure and minimal models as discussed in \cref{r:minmodelequalacycliclcosure}. In positive characteristic, \cref{introthm2} follows a fortiori from \cref{t:abelian}. 

\cref{introthm2} should be contrasted with \cite[Theorem~C]{Avramov/Halperin:1987} which shows, regardless of the residual characteristic, that when $\vp$ has finite projective dimension the only instance $\pi(F^\vp)$ is abelian is when $\vp$ is complete intersection. Hence, without the assumption $\vp$ has finite projective dimension, \cref{introthm2} provides many examples for which $\pi(F^\vp)$ is abelian and $\vp$ is not complete intersection; cf.\@ \cref{r:abelian}.
 
An initial motivation for this project was to better understand certain homological invariants of quasi-complete intersection maps.  
\cref{t:abelian} calculates the homotopy Lie algebra of a quasi-complete intersection map, showing that it exhibits the following dichotomy:

\begin{introthm}
\label{introthm3}Let $\vp$ be a quasi-complete intersection map with residue field $k$, and set $U=\pi^2(F^\vp)$ and $V=\pi^3(F^\vp)$.  There is an isomorphism of abelian Lie algebras
\[\pi(F^\vp)\cong 
\begin{cases}
  U\oplus V & \ch k = 0\\
  U\oplus V \oplus \bigoplus_{t=1}^\infty\shift^{-2p^t+2}V\oplus\bigoplus_{t=1}^\infty\shift^{-2p^t+1}V & \ch k>0
\end{cases}
\]
and the reduced square operation is given by:
  \begin{enumerate}
    \item If $p\neq2$, the reduced square is trivial on $\pi(F^\vp)$.
    \item If $p=2$, then there is a basis $\mathcal{B}$ of $V$ for which the reduced square on $\bigoplus_{t=1}^\infty\shift^{-2p^t+2}V$ is determined by 
    \[\shift^{-2^t+2}b\mapsto \shift^{-2^{t+1}+1} b\] for $b\in \mathcal{B}$ and extended according to the axioms of a reduced square.
    The reduced square is trivial otherwise.
  \end{enumerate}
\end{introthm}

Suitably interpreted, a converse to \cref{introthm3} holds when $\vp$ is assumed to be weakly-closed; see \cref{cor:qci}.  This provides a homological invariant distinguishing quasi-complete intersection maps.  

A final application of the results above, given in \cref{cor:ci}, establishes a characterization of the complete intersection property in terms of the sequence of deviations $\{\ve_i(\vp)\}$ of $\vp$.  These record the $k$-vector space dimensions of the graded pieces of $\pi(F^\vp)$.  When $\vp$ is a map of finite projective dimension, much work has been done to relate vanishing $\ve_i(\vp)$ for large $i$ to the complete intersection property~\cite{Avramov:1999,Avramov/Halperin:1987,Avramov/Iyengar:2003,Gulliksen:1974,Halperin/1987};  a central motivation for these works was to resolve Quillen's conjecture, which claims all such maps must be complete intersection homomorphisms.  The following theorem notably lacks the requirement that $\vp$ has finite projective dimension, a necessary property of prior rigidity results.

\begin{introthm}\label{introthm4}
  Let $\vp\colon R\rightarrow S $ be a closed local homomorphism with residual characteristic $p>0$.
  The following are equivalent:
  \begin{enumerate}
      \item $\vp$ is complete intersection;
      \item $\ve_i(\vp)=0$ for $i\gg0$;
      \item $\ve_i(\vp)=0$ for $i=2p^t+1$ or $i=2p^t+2$ for some $t\geqslant 1$.
  \end{enumerate}
\end{introthm}
\begin{ack}
The authors thank Benjamin Briggs, Srikanth Iyengar, and Alexandra Seceleanu for helpful comments. They also thank an anonymous referee for a careful reading and their comments. The first author also thanks Mark Walker for funding through NSF grant DMS 1901848. The second author was partially supported by NSF grants DMS 1840190 and DMS 2002173.
\end{ack}
\section{Minimal models and acyclic closures}\label{sec:dg}
Throughout $\vp\colon R\to S$ is a surjective map of commutative noetherian local rings with common residue field $k$, and $A=\{A_i\}_{i\geqslant 0}$ will be a dg $R$-algebra. 

In this article, an arbitrary element $x$ of a graded object $X=\{X_i\}_{i\in \mathbb{Z}}$ will be implicitly taken to be homogeneous and $|x|$ will denote that integer $i$ so that $x\in X_i$.  All dg algebras will be strictly graded-commutative in the sense that $ab=(-1)^{|a||b|}ba$ for all $a,b$ in $A$ and $a^2=0$ when $|a|$ is odd. The defining property making $A$ a \emph{dg}
algebra is that it is equipped with a degree $-1$ differential $\del$ satisfying that Leibniz rule: 
\[
\del(ab)=\del(a)b+(-1)^{|a|}a\del(b)\,.
\] Said differently, $(A,\del)$ is a complex of $R$-modules where the multiplication map $A\otimes_R A\to A$ and algebra unit $R\to A$  are morphisms of complexes. 
\begin{chunk}\label{c:ind}Suppose $A_0$ is the local ring $(R,\m)$. If $\del(A_1)\subseteq \m$, then  the maximal dg ideal $\m_A$
 of $A$ is given by 
\[
(\m_A)_i\coloneqq\begin{cases} \m & i=0 \\
A_i & i>0\\
0 & i<0
\end{cases}\,.
\]This gives rise to the complex of $k$-spaces
\[
\ind A\coloneqq \dfrac{\m_A}{\m +\m_A^2}\,
\] which is called the indecomposable complex of $A$. 
\end{chunk}

\begin{chunk}\label{c:semifree}
We write $A[X]$ to denote a semifree extension of $A$ on the graded set of variables $X=X_1,X_2,X_3,\ldots$ as defined in \cite[Section 2.1]{Avramov:2010}.
That is, each variable of $X_i$ has homological degree $i$: When $i$ is even $X_i$ consists of polynomial variables and when $i$ is odd $X_i$ consists of exterior variables. In particular, $A[X]^\natural$ is the free strictly graded-commutative $A$-algebra on $X$ where $(-)^\natural$ forgets the differential of a differential graded object. Note that when $A=R$, the underlying graded $k$-space of $\ind (R[X])$ is $kX;$ cf.\@ \cref{c:ind}.  
\end{chunk}

\begin{chunk}\label{c:minmodel}
A minimal model for $\vp$ is a factorization of $\vp$ as $
R\to R[X]\xra{\simeq} S$ where $\ind(R[X])$ has zero differential. Note that $\ind(R[X])$ has zero differential if and only if the differential of $R[X]$ is decomposable in the sense that $\del(R[X])$ is contained in $\m _RR[X]+(X)^2.$ By \cite[Section~7.2]{Avramov:2010}, minimal models always exist and are unique up to an isomorphism of dg algebras. Their construction begins with the Koszul complex $R[X_1]$ with the set of exterior variables $X_1$ mapping bijectively to a minimal set of generators for $\ker \vp$.   The construction proceeds inductively by adjoining to the dg algebra $R[X_{\leqslant i-1}]$ a set of variables $X_i$ mapping bijectively to a minimal generating set for the homology module ${\rm H}_{i-1} (R[X_{\leqslant i-1}])$; to guarantee that the direct limit of this procedure $R[X]$ is a semifree extension, the sets $X_i$ alternate between polynomial and exterior variables as described in \cref{c:semifree}.
We will slightly abuse terminology and call $R[X]$ the minimal model for $\vp$.
\end{chunk}

\begin{chunk}
\label{c:classiclifting}
Fix a semifree extension $A[X]$ of $A$. It is well known that $A[X]$ enjoys the following lifting property: given a dg $A$-algebra map $f\colon A[X]\to C$ and a surjective quasi-isomorphism of dg $A$-algebras $g\colon B\xra{\simeq} C$, then there exists a dg $A$-algebra map $\tilde{f}\colon  A[X]\to B$, extending $f$, with $g\tilde{f}=f$ that is unique up to $A$-linear homotopy; see, for example, \cite[Proposition~2.1.9]{Avramov:2010}. As a consequence, whenever $R[X]$ is a minimal model of $\vp$ and $B\xra{\simeq} S$ is a dg $R$-algebra resolution, there exists a quasi-isomorphism $R[X]\xra{\simeq}B$ compatible with the augmentation maps to $S$. 
\end{chunk}

Our analysis in \cref{s:comparison} relies on the following general lifting property for semifree extensions; its proof is implicitly contained in the proof of \cite[Proposition~2.1.9]{Avramov:2010}. We sketch the argument below.

\begin{lemma}
\label{l:lifting}
Let $\alpha\colon A\rightarrow B$ be a morphism of dg algebras, and $Z$ a set of cycles in $A$ whose image $\alpha(Z)$ is a set of boundaries in $B$. If $\{b_z: z\in Z\}$ is a collection of elements of $B$ with $\alpha(z)=\del(b_z)$, then there exists a unique map of dg algebras, extending $\alpha$, \[\tilde{\alpha}: A[X\,|\,\del(x_z)=z]\to B \ \text{ with } \ \tilde{\alpha}(x_z)=b_z\]
where $X=\{x_z: |x_z|=|z|+1\}_{z\in Z}$.
\end{lemma}
\begin{proof}
By induction it suffices to check when $Z=\{z\}$, in which case we write $b$ for $b_z$ and $x$ for $x_z.$
As $A[x\mid \del x=z]^\natural$ is free in the category of graded-commutative algebras there is a uniquely defined graded algebra map $\tilde{f}\colon A[x\mid \del x=z]\to B$ with $\tilde{f}|_A=f$ and $x\mapsto b$. It remains to show $\tilde{f}$ commutes with the differentials: 
\begin{align*}
\del^B\tilde{f}\left(\sum a_i x^i\right)
&=\sum \del^Bf(a_i)b^i+(-1)^{|a_i|}f(a_i)\del^B(b^i) \\
&=\sum \del^Bf(a_i)b^i+(-1)^{|a_i|}if(a_i)f(z)b^{i-1} \\
&=\sum f\del^A(a_i)b^i+(-1)^{|a_i|}if(a_i)f(z)\tilde{f}(x)^{i-1} \\
&=\tilde{f}\left(\sum \del^A(a_i)b^i+(-1)^{|a_i|}ia_izx^{i-1}\right) \\
&=\tilde{f}\del^A\left(\sum a_ix^i\right)\,;
\end{align*}the equalities are all evident, and so $\tilde{f}$ is the desired (unique) dg algebra extension of $f$ mapping $x$ to $z$.
\end{proof}

\begin{discussion}\label{d:polyvsdivpowers}
Let $A[x]$ be a semifree extension of $A$ on a single variable of even degree $x$, and suppose $n$ is a positive integer. The element $\del(x)x^{n-1}$ is always a cycle in $A[x]$. It is a boundary if and only if $n$ is invertible in $A_0$, in which case the image of $\frac1nx^n$ under the differential is $\del(x)x^{n-1}$. Tate \cite{Tate:1957} realized that the dependence on the invertiblility of $n$ can be eliminated by adjoining $x$ as a \emph{divided power variable}, as recalled in the sequel.
\end{discussion}

\begin{chunk}\label{c:dividedpower}
We write $A\langle Y\rangle$ for a semifree extension obtained by adjoining a graded set of divided power variables $Y=Y_1,Y_2,Y_3,\ldots$ to $A$; see \cite[Section 6.1]{Avramov:2010} as well as \cite[Chapter~1]{Gulliksen/Levin:1969}. We refer to the variables of $Y$ as $\Gamma$-variables, and to $A\langle Y\rangle$ as a semifree $\Gamma$-extension of $A$. Analogous to \cref{c:semifree}, each $\Gamma$-variable of $Y_i$ has homological degree $i$ with $Y_i$ consisting of divided power variables when $i$ is even, and exterior variables when $i$ is odd.
Recall that when $y\in Y_{\rm{even}}$, its set of divided powers $\{y^{(i)}:|y^{(i)}|=|y|i\}_{i\geqslant0}$ with $y^{(0)}=1$ and $y^{(1)}=y$ satisfy the equalities 
\[
y^{(n)}y^{(m)}=\binom{n+m}{n}y^{(n+m)} \ \text{ and } \ \del(y^{(n)})=\del(y)y^{(n-1)}\,.
\]
These fundamental equations are crucial to the analysis in the next section. 
\end{chunk}
\begin{chunk}
Consider a semifree $\Gamma$-extension $A\langle Y\rangle$. 
We may well-order $Y$ first by homological degree and then by ordering each set $Y_n$. Associated to $A\langle Y\rangle$, with this choice of ordering, is a canonical $A$-linear basis called the \emph{normal $\Gamma$-monomials}; this basis consists of $1$ together with the set of terms of the form \begin{equation}\label{eq:gammamonomial}y_{\lambda_1}^{(i_1)}\cdots y_{\lambda_n}^{(i_n)}\end{equation}
with $y_{\lambda_1}<\ldots<y_{\lambda_n}$ and each $i_j$ is a positive integer, along with the additional constraint that $i_j=1$ when $y_{\lambda_j}$ is of odd degree; see \cite[Section 6]{Avramov:2010} for further details. 

Suppose $A$ is the local ring $R$ with maximal ideal $\m$ and residue field $k$. Let $\m_{R\langle Y\rangle}^{(2)}$ denote the ideal of $R\langle Y\rangle$ generated by all normal $\Gamma$-monomials \cref{eq:gammamonomial} with $i_1+\ldots+i_n\geqslant 2.$ The complex of $\Gamma$-indecomposables of $A\langle Y\rangle$ is the complex of $k$-spaces 
\[
\gind R\langle Y\rangle \coloneqq \dfrac{\m_{R\langle Y\rangle}}{ \m R\langle Y\rangle+\m_{R\langle Y\rangle}^{(2)}}\,;
\]
it is clear that as a graded $k$-vector space $\gind R\langle Y\rangle$ is simply $kY.$
\end{chunk}

\begin{chunk}\label{c:acyclicclosure} 
An acyclic closure for $\vp$ is a factorization of $\vp$ as 
\(
R\to R\langle Y\rangle\xra{\simeq} S
\) where $\gind(R\langle Y\rangle)$ has zero differential. Acyclic closures may be constructed inductively by adjoining $\Gamma$-variables to minimally kill cycles generating the homology modules ${\rm H}_n(R\langle Y_{\leqslant n}))$, as originally described by Tate~\cite{Tate:1957} (see also \cite{Avramov:2010,Gulliksen/Levin:1969}). The procedure is identical to the construction process for minimal models described in \cref{c:minmodel}, except that divided power variables are used in place of polynomial variables.  Uniqueness, up to an isomorphism of dg $\Gamma$-algebras, is contained in \cite[Theorem~1.9.5]{Gulliksen/Levin:1969}. As with minimal models, by a slight abuse of terminology $R\langle Y\rangle$ is referred to as an acyclic closure for $\vp$. 
\end{chunk}

We end this section with a discussion of a property of $\vp$ which allows for the comparison of dg algebra resolutions performed in \cref{s:comparison}.

\begin{definition}\label{d:weaklyclosed}
 We say $\vp$ is \emph{weakly-closed} provided that an acyclic closure $R\langle Y\rangle$ of $\vp$ has decomposable differential; that is, 
 \[
 \del(R\langle Y\rangle) \subseteq \m_R R\langle Y\rangle +(Y)^2\,.
 \]
\end{definition}

\begin{remark}
Let $R\langle Y\rangle$ be a semifree $\Gamma$-extension of $R$.  In analogy to \cref{d:weaklyclosed}, the containment 
\[
 \del(R\langle Y\rangle) \subseteq \m_R R\langle Y\rangle +\m_{R\langle Y\rangle}^{(2)}\,
 \]
is called $\Gamma$-decomposability of the differential. It implies that the differential of $\gind R\langle Y\rangle$ is zero, and if $R\langle Y\rangle$ is a resolution of $S$, it is equivalent to $R\langle Y \rangle$ being an acyclic closure; see, for example, \cite[Lemma~6.3.2]{Avramov:2010}.

This classical notion is similar to the one introduced in \cref{d:weaklyclosed}, but they are distinct.  
A precise description of their difference is provided by \cref{l:comparingindecomposables} and \cref{r:decomposability}.
\end{remark}

\begin{remark}
As any two acyclic closures of $\vp$ are isomorphic as local dg algebras (in fact, as local dg $\Gamma$-algebras) the property of $\vp$ being weakly-closed is independent of choice of acyclic closure. If $\mathbb{Q}\subseteq R$, then $\vp$ is trivially weakly-closed. Indeed, in this case $\m_{R\langle Y\rangle}^{(2)}=(Y)^2$; see for example \cref{r:minmodelequalacycliclcosure}. The previous remark therefore guarantees  $\vp$ is weakly-closed.

Nontrivial examples of such maps---regardless of characteristic considerations---are closed maps, defined in the sequel. See also the discussion in \cref{r:question}.
\end{remark}

\begin{chunk}\label{c:closedexamples}
Following \cite[1.3]{Avramov/Iyengar:2005}, we say $\vp$ is closed if some (equivalently, every) acyclic closure $R\langle Y\rangle$ of $\vp$ is minimal as a complex of $R$-modules. That is, $\del(R\langle Y\rangle)\subseteq \m_R R\langle Y\rangle$. Clearly any closed homomorphism is weakly-closed. Examples of closed morphisms are plentiful:
\begin{enumerate}
 \item The augmentation map to the residue field is closed. This is the content of a celebrated theorem of Gulliksen and Schoeller~\cite{Gulliksen:1968,Schoeller:1967}
 \item More generally, any large homomorphism---that is, $\vp$ satisfying that the induced map on Tor algebras 
 \( \Tor^R(k,k)\to \Tor^S(k,k)
 \)
 is surjective---is closed by a theorem of Avramov and Rahbar-Rochandel; see
 \cite[Theorem~2.5]{Levin:1980}. A prominent example of a family of large homomorphisms is provided by algebra retracts, in the sense that there exists a ring map from $S$ to $R$ which is right inverse to $\vp$. 
 \item Quasi-complete intersection maps, defined and studied by Avramov, Henriques and \c{S}ega, are closed; cf.\@ \cite[1.6]{Avramov/Henriques/Sega:2013}. Recall $\vp$ is quasi-complete intersection if in the acyclic closure $R\langle Y\rangle$ of $\vp$ we have $Y_i=\emptyset$ for all $i\geqslant 3.$  An elementary example is when $\vp$ is a complete intersection homomorphism---in the usual sense that the kernel of $\vp$ is generated by an $R$-regular sequence---in which case $Y_2=\emptyset$.  The augmentation $R\rightarrow k$, when $R$ is a complete intersection ring, is an additional example of a quasi-complete intersection;  when $R$ is singular the augmentation is not a complete intersection~\cite{Tate:1957}.  An additional concrete example is the quotient $R\rightarrow R/(r)$ when $r$ is an exact zero divisor as introduced in \cite{Henriques/Sega:2011}.
\end{enumerate}
\end{chunk}

\begin{remark}\label{r:quillen}
Quasi-complete intersection maps are central to a conjecture due to Quillen on the cotangent complex~\cite{Quillen:1970}. The latter has prompted numerous investigations into the former; see, for example, \cite{Avramov/Henriques/Sega:2013,Bergh/Celikbas/Jorgensen:2014,Henriques/Sega:2011,Kustin/Sega:2018,Lutz:2016,Majadas/Rodicio:2010,Rodicio:1992,Windle:2017}.
\cref{introthm3} from the introduction, as well as \cref{cor:qci}, are two of the main results of this article as they reveal further structural restrictions on quasi-complete intersection maps. We hope these insights can lead to  progress towards settling Quillen's conjecture, as analogous restrictions for complete intersection maps were fundamental for resolving an analogous conjecture of Quillen and understanding these maps more generally.
\end{remark}

\section{Divided powers in prime residual characteristic}
This section contains various facts involving divided power algebras and binomial coefficients that are needed for our analysis in \cref{s:comparison}.
Foundational material on divided power algebras is contained in \cite{Cartan:1956,Roby:1980}, which includes some of the formulas in this section; we provide a self-contained treatment below.
Throughout this section $p$ will be a fixed prime number. 
\begin{chunk}\label{c:lucas}
Recall that every integer $n$ may be written uniquely in its base $p$ expansion as \[n_0p^0+n_1p^1+\dots+n_lp^l\]
where each $n_i$ is a nonnegative integer strictly less than $p$; in this case we write $n=[n_0n_1\dots n_l]_p$. We use $\equiv_p$ to denote equivalence of integers modulo $p$. 

In the notation above, we recall the following classical theorem of Lucas~\cite{Lucas:1878}: \emph{If $n=[n_0\dots n_l]_p$ and $m=[m_0\dots m_l]_p$, then} \[\binom{n}{m}\equiv_p\binom{n_0}{m_0}\cdots\binom{n_l}{m_l}.\]
In applying this theorem, note the convention that $\binom{n}{m}=0$ whenever $m>n$.
\end{chunk}

\begin{notation}\label{c:coefficient}
Let $y$ be a divided power variable over a ring $R$. Using \cref{c:dividedpower} it follows easily that for any positive integer partition of $n=n_0+n_1+\ldots+n_l$, there is the equality \[y^{(n_0)}\cdots y^{(n_l)}=\binom{n}{n_0,n_1,\dots,n_l}y^{(n)} \ \text{ where } \ \binom{n}{n_0,n_1,\dots,n_l}=\dfrac{n!}{n_0!n_1!\cdots n_l!}\,.\] 

Define the following nonnegative integers \[b_{t,m}=\binom{tp^m}{\,\underbrace{p^m,p^m,\dots,p^m}_{t\text{-times}}}\ \text{ and } \ c_n\coloneqq\binom{n}{n_0,n_1p,\dots,n_lp^l} \] for each $m,n,t\in \mathbb{N},$ where $n=[n_0\dots n_l]_p.$ These arise as the coefficients in the following multiplications: 
\begin{align*}
 (y^{(p^m)})^{t}&=b_{t,m}y^{(tp^m)}\\
 y^{(n_0)}y^{(n_1p)}\cdots y^{(n_lp^l)}&=c_ny^{(n)} \, \text{ where } \, n=[n_0\dots n_l]_p\,.
\end{align*}

The coefficient $b_{p-1,m}$ arises sufficiently often that we will denote it $b_{m}$ for simplicity.

\end{notation}
\begin{lemma}\label{l:dividedproduct}\label{l:dividedpower}
Adopting \cref{c:coefficient}, there are equivalences 
\begin{center}
\begin{tabular}{c c c c c c c}
 $b_{t,m}\equiv_p t!$ & & & $b_{m}\equiv_p -1$ & & & $c_n\equiv_p 1$
\end{tabular}
\end{center}
for any nonnegative integers $m$ and $n$ and for $t<p$. In particular, if $R$ has residual charactersitic $p$, then $b_{t,m}$, $b_m$, and $c_n$ are units in $R$.
\end{lemma}
\begin{proof}
For the first equivalence
\[b_{t,m}=\binom{p^m}{p^m}\binom{2p^m}{p^m}\cdots\binom{tp^m}{p^m}\equiv_p t!\]
where the equivalence follows from Lucas' theorem, recalled in \cref{c:lucas}. Wilson's theorem~\cite{Hardy/Wright:2008} says $(p-1)!\equiv_p-1$, giving the second equivalence. 

For the third equivalence, we induct on $l$; the base case, when $l=0$, is trivial. Next assume $l>0$ and set $n'=[n_0\dots n_{l-1}]_p$. Observe that 
\[   
c_n=c_{n'}\binom{n}{n_lp^l}\equiv_p1
\]
where the first equality uses $n = n'+n_l p^l$ and the equivalence holds by Lucas' theorem and induction.  
\end{proof}
\begin{remark}
The precise formulas for the coefficients specified in \cref{c:coefficient} are primarily needed in \cref{s:comparison} to handle the case when $R$ has mixed characteristic.  When $R$ itself has characteristic $p$, the arguments in the proof of \cref{l:dividedproduct} show that $b_m=-1$ and $c_n=1$. That is, \[(y^{(p^m)})^{p-1}=-y^{((p-1)p^m)} \ \text{ and } \ y^{(n)}=y^{(n_0)}y^{(n_1p)}\cdots y^{(n_lp^l)}\,\] where $n=[n_0\dots n_l]_p$.
\end{remark}

\begin{notation}
 Suppose $(R,\m,k)$ is local and $k$ has characteristic $p>0.$ For any semifree $\Gamma$-extension $R\langle Y\rangle$ of $R$, we let
$kY^{(p^\infty)}$ denote the subcomplex of $\ind R\langle Y\rangle$ consisting of cycles of the form $y^{(p^i)}$ with $y\in Y_{\rm{even}}$ and $i>0.$
\end{notation}

\begin{lemma}\label{l:comparingindecomposables}
 Suppose $(R,\m,k)$ is local and $k$ has characteristic $p>0.$ For any semifree $\Gamma$-extension $R\langle Y\rangle$ of $R$,
 there is an exact sequence of complexes of $k$-spaces
 \[0\rightarrow kY^{(p^\infty)}\to \ind R\langle Y\rangle\rightarrow \gind R\langle Y\rangle \rightarrow 0\,.\]
\end{lemma}

\begin{proof}
Set $A=R\langle Y\rangle$. The containment of subcomplexes 
\begin{equation}
\label{e:containment}\m+\m_A^2=\m A +\m_A^2\subseteq \m A+\m_{A}^{(2)}\end{equation} induces a surjection $ \ind A\rightarrow \gind A.$
It suffices to examine when $y^{(n)}$ is zero in $\ind A$ with $y\in Y_{\textrm{even}}$, and $n> 1$ with $n=[n_0\dots n_l]_p$.

When $n$ is not a power of $p$, there are two cases. First if at least two $n_i$ are nonzero, then using \cref{c:coefficient} and \cref{l:dividedpower}, we obtain the equality
\[
y^{(n)}=c_n^{-1}y^{(n_0)}\ldots y^{(n_lp^l)}\,.
\] The second case is that $n_i=0$ for $l>i$ and $n_l>1$. Here 
\[
y^{(n)}=y^{(n_lp^l)}=b_{n_l,l}^{-1}(y^{(p^l)})^{n_l}
\] where the second equality again uses \cref{c:coefficient} and \cref{l:dividedpower}. Hence, in either case, we have shown $y^{(n)}$ is zero in $\ind A$.

Now we show that when $n$ is a power of $p$, say $p^i$, then $y^{(n)}$ is not zero in $\ind A.$ Assuming to contrary, in $A\otimes_R k$ one can write $y^{(n)}$ as a $k$-linear combination of $\Gamma$-monomials of the form
\[
y^{(i_1)}\cdots y^{(i_t)}
\] with $i_1+\ldots +i_t=n$ and $t>1.$
Set $m=i_1+\dots+i_{t-1}$. Using \cref{l:dividedproduct}, we have \[y^{(i_1)}\cdots y^{(i_t)}=\binom{m}{n_1,\dots,n_{t-1}}\binom{n}{i_t}y^{(p^i)}\] and since $t\geqslant 1$, it follows that $i_t<n$. Observe that $\binom{n}{i_t}\equiv_p0$ by Lucas' theorem. As a consequence 
$y^{(i_1)}\cdots y^{(i_t)}$ equals zero in $A\otimes k$. 
\end{proof}

\begin{remark}\label{r:decomposability}
A simple consequence of \cref{l:comparingindecomposables}, that is easy to see independently, is that decomposability of the differential of $R\langle Y\rangle$ implies $\Gamma$-decomposability of its differential. The main content of \cref{l:comparingindecomposables} is that the converse holds if no summand of the form $y^{(p^i)}$, for $y$ in $Y_{\textrm{even}}$ and positive integer $i$, appears as a summand in the differential of any element of $R\langle Y\rangle.$
\end{remark}

\section{Comparison map}\label{s:comparison}
Throughout, we fix a surjective homomorphism of commutative local noetherian rings $\vp\colon R \to S$ with common residue field $k$. Let $R[X]\xra{\simeq} S$ be a minimal model for $\vp$ and let $R\langle Y\rangle\xra{\simeq} S$ denote an acyclic closure for $\vp$. 
\begin{chunk}
By \cref{c:classiclifting}, there exists a quasi-isomorphism of dg $R$-algebras $R[X]\xra{\simeq} R\langle Y\rangle$ which is unique up to homotopy and compatible with the augmentations to $S$. In this section we construct an explicit representative of this map, which will denoted by $\comp^\vp$ and referred to as the \emph{comparison map of }$\vp$. As $\comp^\vp$ is a morphism of local dg algebras, along with an analogous containment to \cref{e:containment}, the comparison map induces a map on complexes of $k$-spaces 
\[
\ind(\comp^\vp)\colon \ind (R[X])\to \gind( R\langle Y\rangle)\,.
\]
\end{chunk}
\begin{remark}\label{r:minmodelequalacycliclcosure}
When the characteristic of $k$ is further assumed to be zero, it is then standard that $R[X]$ and $R\langle Y\rangle$ are isomorphic, and easy to see that $\comp^\vp$ is itself an isomorphism of dg algebras which induces the isomorphism $\ind(\comp^\vp)$.  

Indeed, if $R$ is local and $\ch k =0$, then $R$ contains a copy of $\mathbb{Q}$. In this setting, for any cycle $z$ in a dg $R$-algebra $A$, the extensions $A[x\,|\,\del(x)=z]$ and $A\langle x\,|\,\del(x)=z\rangle$ coincide when $|x|$ is odd. When $|x|$ is even,  the map $A[x\,|\,\del(x)=z]\to A\langle x\,|\,\del(x)=z\rangle$ given by \[\sum_{i\geq 0}a_ix^i\mapsto \sum_{i\geq 0}i!a_ix^{(i)}\] establishes an isomorphism of dg $R$-algebras; the fact this map is invertible is the only place that $\mathbb{Q}\subseteq R$ is used.  Thus $R[X]$ and $R\langle Y\rangle$ are each a direct limit of isomorphic extensions and $\comp^\gamma$ is the direct limit of these isomorphisms.

The goal of this section is to provide insight on the situation when $k$ has prime characteristic $p>0$, which we assume for the remainder of the section.
\end{remark} 

In this section, we prove \cref{introthm1} which is recast, in the notation set above, in \cref{t:comparison}. It is an immediate consequence of \cref{t:refinement} presented at the end of the section. Recall that for a graded object $V$, the $i$-fold suspension of $V$ is $\shift^iV$ where $(\shift^i V)_{j}=V_{j-i}.$  

\begin{corollary}\label{t:comparison}
Suppose $\vp\colon R\to S$ is a surjective map of local rings whose common residue field $k$ has characteristic $p>0$. If $\vp$ is weakly-closed, then the comparison map $\comp^\vp$ is surjective and induces the following exact sequence of graded $k$-spaces 
 \[
 0\to kY^{(p^\infty)}\oplus\shift kY^{(p^\infty)} \to \ind (R[X])\xra{\ind(\comp^\vp)} \gind (R\langle Y\rangle)\to 0.\qed
 \]
\end{corollary}

\begin{construction}\label{const:comparison}
Let $A$ be a dg $R$-algebra and $z\in A_{\rm{odd}}$ be a cycle. By \cref{l:lifting}, there is a canonical map of dg $A$-algebras
\[
\comp\colon A[x\,|\,\del x=z]\to A\langle y\,|\,\del y=z\rangle
\] completely determined by $x\mapsto y$. 
We define a semifree extension \[A[x]\hookrightarrow A[\{x_i,x_{i+1}'\}_{i\geqslant 0}]\] and a map $\comp_z(A): A[\{x_i,x_{i+1}'\}_{i\geqslant 0}]\to A\langle y\rangle$ of dg $A$-algebras extending $\comp$. Recall the coefficients $b_{t,m}$ and $c_n$ defined in \cref{c:coefficient}.  For notational convenience, set \begin{equation}
\label{eq:di}
d_{i} = c_{p^i-1}\prod_{j=0}^{i-1}(b_jc_{p^j-1}^{p-1})^{p^{i-1-j}}\,.
\end{equation}

Set $x_0\coloneqq x$ and by induction assume we have constructed $A[\{x_i,x_{i+1}'\}_{0\leqslant i\leqslant n}]$, denoted $A(n)$,
with differential determined by \[\del(x_i)=zx_0^{p-1} x_1^{p-1}\ldots x_{i-1}^{p-1} \ \text{ and } \ \del(x_i')=x_{i-1}^{p}-px_i,
\] and a dg algebra map $\comp_z(n): A(n)\to A\langle y\rangle$ determined by $\comp_z(n)(x_0)=y$ and for all $i\geqslant 1$: 
\[\comp_z(n)(x_i)=d_iy^{(p^i)} \ \text{ and } \ \comp_z(n)(x_i')=0\,.\]

By \cref{l:dividedpower}, each $d_i$ is a product of units and hence is a unit of $R$.
A tedious direct calculation, using \cref{l:dividedpower}, yields 
\[
\comp_z(n)(zx_0^{p-1}\cdots x_n^{p-1})=d_{n+1}\del(y^{(p^{n+1})})\,.
\]

 Also, as $z$ has odd degree, it follows that $zx_0^{p-1} x_1^{p-1}\ldots x_{n}^{p-1}$ is a cycle of degree $(|x|p^{n+1}-1)$.
Therefore by \cref{l:lifting}, there is a unique dg $A$-algebra map 
\begin{equation}\label{inductivemap}
A(n)[x_{n+1}\,|\,\del x_{n+1}=zx_0^{p-1}x_1^{p-1}\ldots x_n^{p-1}]\to A\langle y\rangle
\end{equation}
extending $\comp_z(n)$ with $x_{n+1}\mapsto d_{n+1}y^{(p^{n+1})}.$

Under this extension the image of $x_{n}^{p}-px_{n+1}$ is a cycle of degree $|x|p^{n+1}$. Using the definitions of $d_n$ and $c_i$, the image of this cycle under \cref{inductivemap} is 
\begin{align*}
 d_n^p(y^{(p^{n})})^{p}-pd_{n+1}y^{(p^{n+1})}&= \left(b_nd_n^p\binom{p^{n+1}}{p^n}-pd_{n+1}\right)y^{(p^{n+1})}\\
 &=\left(\binom{p^{n+1}}{p^n}\frac{c_{p^n-1}}{c_{p^{n+1}-1}}-p\right)d_{n+1}y^{(p^{n+1})}\\
 &=0\,.
\end{align*}
 Hence $\comp_z(n)$ can be further extended, again applying \cref{l:lifting}, to a dg $A$-algebra map 
$\comp_z(n+1): A(n+1)\to A\langle y\rangle$
 where
\[A(n+1)\coloneqq A[\{x_i,x_{i+1}'\}_{0\leqslant i\leqslant n+1}\,|\,\del x_{i+1}=zx_0^{p-1}\ldots x_i^{p-1}, \ \del x_{i+1}'=x_i^p-px_{i+1}]
\] with $\comp_z(n+1)(x_i)=d_iy^{(p^i)}$ and $\comp_z(n+1)(x_i')=0$ for each $i$, which completes the induction. 

Now taking the colimit of the maps $ \comp_z(n)$, we obtain the dg $A$-algebra map $\comp_z(A): A[\{x_i,x_{i+1}'\}_{i\geqslant 0}]\to A\langle y\rangle$ extending $\comp$ with \[\comp_z(A)(x_i)=d_iy^{(p^i)} \ \text{ and } \comp_z(A)(x_i')=0\,.\]
\end{construction}

\begin{lemma}\label{l:adjoineven}
In the notation of \cref{const:comparison}, the dg $A$-algebra map \[\comp_z(A): A[\{x_i,x_{i+1}'\}_{i\geqslant 0}]\to A\langle y\rangle \] is a quasi-isomorphism. 
\end{lemma}
\begin{proof}
As the regular element $x_i^p-px_{i+1}$ is the boundary of the corresponding $x_{i+1}'$, by iteratively applying \cite[Theorem~3]{Tate:1957} it follows that the canonical map 
\[
A[\{x_i,x_{i+1}'\}_{i\geqslant 0}]\xra{\simeq}A[\{x_i\}_{i\geqslant 0}]/(\{x_i^p-px_{i+1}\}_{i\geqslant 0})
\] is a quasi-isomorphism. 
Next observe that $\comp_z(A)$ factors through the canonically induced map 
\[
\overline{\comp}: A[\{x_i\}_{i\geqslant 0}]/(\{x_i^p-px_{i+1}\}_{i\geqslant 0})\to A\langle y\rangle
\] with $x_i\mapsto d_iy^{(p^i)}$ for each $i\geqslant 0$; cf.\@ \cref{const:comparison}\cref{eq:di} for the definition of $d_i$ which is defined in terms of the coefficients introduced in \cref{c:coefficient}. 

Let $n$ be a nonnegative integer and write $n=[n_0n_1\ldots n_l]_p$. Observe that 
\begin{align*}
 \overline{\comp}(x_{0}^{n_0}x_1^{n_1}\cdots x_{l}^{n_l})&=\prod_{i=1}^l (d_{i}y^{(p^i)} )^{n_i}\\
 &=\prod_{i=1}^l d_{i}^{n_i}b_{n_i,i}y^{(n_ip^i)} \\
 &=\left(\prod_{i=1}^l d_{i}^{n_i}b_{n_i,i}\right)c_ny^{(n)}
\end{align*}the first equality holds as $\overline{\comp}$ is an algebra map and the other equalities use \cref{c:coefficient}. Also, an $A$-linear basis for $A[\{x_i\}_{i\geqslant 0}]/(\{x_i^p-px_{i+1}\}_{i\geqslant 0})$
is \[\{
x_{0}^{n_0}x_1^{n_1}\cdots x_{l}^{n_l}: 0\leqslant n_i<p
\}\] and it is standard that an $A$-linear basis for $A\langle y\rangle$ is $\{y^{(n)}\}_{n\geqslant 0}$, the divided powers of $y$. By \cref{l:dividedpower}, the nonnegative integers \[\left(\prod_{i=1}^l d_{i}^{n_i}b_{n_i,i}\right)c_n\]
are units in $R$ and so $\overline{\comp}$ is an isomorphism of dg $A$-algebras, which finishes the proof of the lemma once recalling $\comp_z(A)$ factors as 
\[
A[\{x_i,x_{i+1}'\}_{i\geqslant 0}]\xra{\simeq}A[\{x_i\}_{i\geqslant 0}]/(\{x_i^p-px_{i+1}\}_{i\geqslant 0})\xra{\overline{\comp}} A\langle y\rangle.\qedhere
\]
\end{proof}

\begin{theorem}
\label{t:refinement}
Suppose $\vp\colon R\to S$ is a surjective morphism of local rings of residual characteristic $p>0$. If $\vp$ is weakly-closed, then the minimal model $R[X]$ for $\vp$ has the form 
\[
 R[X(0),X(1),X'(1),X(2),X'(2),\ldots ]
\] where 
\begin{enumerate}
 \item $X(0)=\{x_0(y): y\in Y\}$ with $|x_0(y)|=|y|$,
 \item for $i\geqslant 1$, $X(i)=\{x_i(y): y\in Y_{\rm{even}}\}$ with $|x_i(y)|=|y|p^i$,
 \item for $i\geqslant 1$, $X'(i)=\{x_i'(y): y\in Y_{\rm{even}}\}$ with $|x_i'(y)|=|y|p^i+1$,
\end{enumerate} 
and there exists a surjective quasi-isomorphism $\comp^\vp\colon R[X]\rightarrow R\langle Y\rangle$ determined by the formulas
\[
\comp^\vp(x_i(y))=d_iy^{(p^i)} \ \text{ and }\comp^\vp(x_i'(y))=0
\]
where $d_i$ is the unit defined in \cref{const:comparison}\cref{eq:di}.
Furthermore, the differential of $R[X]$ is given by 
\[
 \del(x_{i}(y))=\tilde{z}\prod_{j=0}^{i-1}x_{j}(y)^{p-1} \ \text{ and } \ 
 \del(x_{i}'(y))=x_{i-1}(y)^{p}-px_{i}(y)
\]
where $\tilde{z}$ is a cycle lifting the cycle $\del(y)$ along $\comp^\vp$.
\begin{proof}
Assume, by induction, we have constructed a surjective quasi-isomorphism of dg $R$-algebras 
\[
\comp(n): R[X(0)_{\leqslant n}, X(1)_{\leqslant n}, X'(1)_{\leqslant n},\ldots ]\to R\langle Y_{\leqslant n}\rangle
\] with the desired properties for some $n\geqslant 0$, and let $A(n)$ denote the source of this map.

Let $z$ be a cycle in $R\langle Y_{\leqslant n}\rangle$ of homological degree $n$. Since $\comp(n)$ is a surjective quasi-isomorphism, there exists a cycle $\tilde{z}$ in $A(n)$ with $\comp(n)(\tilde{z})=z.$ Furthermore, by \cref{r:decomposability} the assumption that $\vp$ is weakly-closed guarantees no summand of $z$ is of the form $y^{(p^i)}$ for $y\in Y_{\text{even}}$, which implies that $\tilde{z}$ can be chosen to be in $\m_R A(n) +\m_{A(n)}^2$.

When $n$ is even, $\comp(n)$ extends to the quasi-isomorphism of dg $R$-algebras
\[
A(n)\langle x \,|\,\del x=\tilde{z}\rangle \xra{\simeq} R\langle Y_{\leqslant n}\rangle \langle y\,|\,\del y=z\rangle;
\] cf.\@ \cite[Lemma~7.2.10]{Avramov:2010}. As $x$ is an exterior variable, $A(n)[x\,|\,\del x = \tilde{z}]$ and $A(n)\langle x\rangle$ coincide. 

Now assume $n$ is odd. In this case we obtain a surjective quasi-isomorphism 
\begin{equation}
\label{firstquism}
A(n)[\{x_i,x'_{i+1}\}_{i\geqslant 0}] \xra{\simeq} A(n)\langle y\,|\,\del y=\tilde{z}\rangle,
\end{equation} with $\del x_i=\tilde{z}x_0^{p-1}\ldots x_{i-1}^{p-1}$ and $\del x_{i+1}'=x_i^{p-1}-px_{i+1}$, using \cref{l:adjoineven}. Furthermore, another application of \cite[Lemma~7.2.10]{Avramov:2010}, extends $\comp(n)$ to a surjective quasi-isomorphism 
\begin{equation}\label{secondquism}
A(n)\langle y\,|\,\del y=\tilde{z}\rangle\xra{\simeq} R\langle Y_{\leqslant n}\rangle \langle y\,|\,\del y=z\rangle.
\end{equation} Composing the surjective quasi-isomorphisms from \cref{firstquism} and \cref{secondquism} yield another one 
\[
A(n)[\{x_i,x'_{i+1}\}_{i\geqslant 0}] \xra{\simeq} R\langle Y_{\leqslant n}\rangle \langle y\,|\,\del y=z\rangle.
\] Repeating this for each cycle of degree $n$, extends $\comp(n)$ to a surjective quasi-isomorphism 
\[
\comp(n+1): R[X(0)_{\leqslant n+1}, X(1)_{\leqslant n+1}, X'(1)_{\leqslant n+1},\ldots ]\to R\langle Y_{\leqslant {n+1}}\rangle.
\] Taking the colimit of these maps yields the desired surjective quasi-isomorphism 
\[
\comp^\vp\colon R[X(0),X(1),X'(1),X(2),X'(2),\ldots ]\xra{\simeq} R\langle Y\rangle 
\] satisfying all of the desired properties. 
\end{proof}
\end{theorem}

\section{The homotopy Lie algebra} 
\label{sec:HLA}
Fix a surjective map $\vp\colon R \to S$ of local rings with common residue field $k$. First we recall the homotopy Lie algebra $\pi(F^\vp)$ introduced by Avramov. Its structure reflects interesting ring-theoretic properties of $\vp$; 
suitable references include~\cite{Avramov:1984,Avramov:2010,Avramov/Halperin:1986,Avramov/Halperin:1987}. 

\begin{remark}\label{r:HLAHistory}
Homotopy Lie algebras in local algebra play an analogous role to those arising in rational homotopy theory.  Inspired by their utility in the study of finite-type simply-connected CW complexes, Avramov, Halperin, Roos, and others initiated a study of the corresponding Lie algebras in local algebra.  A thorough dictionary between the Lie algebras in rational homotopy theory and those in local algebra may be found in~\cite{Avramov/Halperin:1986}.  Notable applications in commutative algebra include resolving conjectures of Quillen regarding cotangent cohomology~\cite{Avramov/Halperin:1987,Avramov/Iyengar:2010}, as well as settling a conjecture of Vasoncelos on the conormal module~\cite{Briggs:2020}.  Their construction is described in the sequel.
\end{remark}

\begin{chunk}\label{c:HLA}
Let $R[X]\xra{\simeq} S$ be a minimal model for $\vp$, and let $k[X]$ denote $k\otimes_R R[X].$ Also let $kX^n$ be the $k$-linear space generated by all monomials in $X$ of degree $n$; the space $kX^1=kX$ is canonically isomorphic to $\ind (R[X])$ (as defined in \cref{c:ind}), and so these will be naturally identified. The assumption that $R[X]$ has decomposable differential yields the decomposition 
\[
\del^{k[X]}=\del^{[2]}+\del^{[3]}+\ldots 
\] with $\del^{[i]}|_{kX}:kX\to kX^{i}$. The equality $\del^{[2]}\del^{[2]}=0$ defines a graded Lie algebra structure on $(\shift kX)^*$ as recalled below. In what follows, $(-)^*$ denotes $k$-linear duality, and we equip $(\shift kX)^*$ with the dual basis of functionals $\shift x^*$ where 
\[
\shift x^*(\shift x')=\begin{cases}
 1 & x=x'\\
 0& x\neq x'
\end{cases}
\] for $x,x'\in X.$

As a graded $k$-space the homotopy Lie algebra of $\vp$, denoted $\pi(F^\vp)$, is $(\shift kX)^*$. The Lie structure on $\pi(F^\vp)$ is defined using $\del^{[2]}$ along with
a fixed well-ordering of $X$, as usual first ordered by homological degree. Namely, writing 
\[
\del^{[2]}(x_l)=\sum_{i<j}q_{ij}^l x_ix_j+\sum_i q_i^l x_i^2\,,
\] the Lie bracket and reduced square on $(\shift kX)^*$ are defined, with $j>i$, as 
\[
[\shift x_j^*,\shift x_i^*]=\sum_l q^l_{ij} \shift x_l^* \ \text{ and } \ (\shift x_i^*)^{[2]}=-\sum_l q_i^l\shift x_l^*\,.
\]
\end{chunk}

\begin{discussion}\label{functoriality/betternotation}
There are two functors involved in the formation of the homotopy Lie algebra of a surjective local homomorphism, which explains our choice of notation above. Namely, the first functor $F$ associates to the surjective local homomorphism $\vp\colon R\to S$ the (derived) fiber 
\[
F^\vp\coloneqq k\lotimes_R S\simeq k[X]\,
\] in the notation of \cref{c:HLA}. The functor $\pi$ is naturally defined on the category of semifree dg $k$-algebras with decomposable differential as discussed in \emph{loc.\@ cit.}; the target category is that of graded Lie algebras over $k$ of finite type. 

One can also regard $\pi$ as the functor which associates to a semifree dg $k$-algebra with decomposable differential the homology of the dg Lie algebra of $\Gamma$-derivations of an acyclic closure of $k$ over $A$; this is the perspective taken in \cite[Chapter~10]{Avramov:2010}, though this will be less relevant for what follows. 
\end{discussion}

We now arrive at one of the main applications of \cref{t:refinement} which shows the homotopy Lie algebra of a non-complete intersection weakly-closed map always has an infinitely generated abelian Lie subalgebra.

\begin{theorem}\label{t:abelian}
Let $\vp\colon R\rightarrow S$ be a surjective local map with prime residual characteristic $p>0$ and let $R\langle Y\rangle$ be an acyclic closure of $\vp$. If $\vp$ is weakly-closed, then there is an isomorphism of graded Lie algebras 
\[
\pi(F^\vp)\cong L\times L^\infty \ \text{ with } \ [\pi(F^\vp),L^\infty]=0
\]
where $L\cong (\shift kY)^*$ and $L^\infty\cong (\shift kY^{(p^\infty)}\oplus \shift^2 kY^{(p^\infty)})^*$ as $k$-spaces. Furthermore, when $p>2$ the reduced square on $L^\infty$ is trivial. When $p=2$, the reduced square of $(\shift y^{(2^i)})^*\in L^\infty$ is given by $(\shift^2 y^{(2^{i+1})})^*$ and trivial otherwise.
\end{theorem}
\begin{proof}
Take $R[X]=R[X(0),X(1),X'(1),\dots]$ as described in \cref{t:refinement} to be a minimal model of $\vp$. 
Define the $k$-subspaces of $\pi(F^\vp)$:
\[
L\coloneqq (\shift kX(0))^* \ \text{ and } \ L^\infty\coloneqq \left(\bigoplus_{i=1}^\infty \shift kX(i)\oplus \shift^2 kX'(i)\right)^*\,.
\]
The asserted isomorphisms on $k$-spaces are induced by the obvious isomorphisms
\[
kY\cong kX(0) \ \text{ and } \ kY^{(p^\infty)}\cong \bigoplus_{i=1}^\infty kX(i)
\] given by $y\mapsto x_0(y)$ for each $y\in Y$ and $y^{(p^j)}\mapsto x_j(y)$ for each $y\in Y_{\textrm{even}}$ and any $j\geqslant 1$, respectively. 

When $p\neq 2$, a direct calculation using \cref{t:refinement} shows that for $i>0$ no element of $X(i)$ or $X'(i)$ appears as a summand of any element in the image of $\del^{[2]}$ and that $\del^{[2]}(X(i)\cup X'(i))=0$. Using these facts and \cref{c:HLA} to compute the Lie bracket and reduced square on $\pi(F^\vp)$ it follows that $L$ and $L^\infty$ are Lie subalgebras of $\pi(F^\vp)$ with the property $[\pi(F^\vp),L^\infty]=0$ with trivial reduced square on $L^\infty$.

Finally, when $p=2$, a similar calculation together with the equality 
\[
\del^{[2]}(x_{i+1}'(y))=x_i(y)^2
\] 
yields the last conclusion. 
\end{proof}

The next corollary, along with \cref{t:abelian}, establishes \cref{introthm2} from the Introduction. 
\begin{corollary}\label{cor:abelian}
If $\vp$ is closed, then $\pi(F^\vp)$ is an abelian Lie algebra. Furthermore, $\pi(F^\vp)$ is abelian as a restricted Lie algebra whenever the residue field has characteristic different from 2.
\end{corollary}
\begin{proof}When $R$ has characteristic zero, \cref{r:minmodelequalacycliclcosure} and the definition of a closed map ensure that $\del^{[2]}=0$ on all of $X$, so the formula for the bracket in \cref{c:HLA} yields that $\pi(F^\vp)$ is abelian as a restricted Lie algebra. 

Next, consider the case when $\ch k=p>0$. The assumption that $\vp$ is closed guarantees $\del^{[2]}$ vanishes on any variable belonging to $X(0)$,  in the notation of \cref{t:comparison}. Hence the Lie algebra $L$ in \cref{t:abelian} is an abelian restricted Lie subalgebra of $\pi(F^\vp)$.
\end{proof}

\begin{remark}\label{r:abelian}
When $\vp$ is a map of finite projective dimension---in the sense that $S$ has finite projective dimension over $R$---the only case that $\pi(F^\vp)$ is abelian is when $\vp$ is complete intersection; cf.\@ \cite[Theorem~C]{Avramov/Halperin:1987}.  \cref{cor:abelian} provides a wide class of maps (not necessarily of finite projective dimension) for which $\pi(F^\vp)$ is abelian; see the examples in \cref{c:closedexamples}. A consequence is that there cannot be a direct analog to \cite[Theorem~C]{Avramov/Halperin:1987} for detecting the quasi-complete intersection property in terms of an abelian Lie algebra structure of $\pi(F^\vp)$. 
\end{remark}

\section{Applications involving deviations and closing remarks}
\label{sec:qci}
Let $\vp\colon R\rightarrow S$ be a surjective local map, and let $k$ denote the common residue field of $R$ and $S$. 
In this section we present  applications involving the deviations of $\vp$; see \cref{c:deviations}, below, for the definition. 

The rigidity and vanishing of the deviations of a local homomorphism have been well-studied; see for example, \cite{Avramov:1999,Avramov/Halperin:1987,Avramov/Iyengar:2003,Briggs:2018}. However, a crucial assumption for many of these works is that $\vp$ is a map of finite projective dimension. The corollaries in this section add to the rigidity results above, without the restriction on maps of finite projective dimension. Instead, our results apply when $\vp$ is a weakly-closed surjective local map with residual characteristic $p>0$; the generic examples in \cref{c:closedexamples} are maps of infinite projective dimension.  

\begin{chunk}\label{c:deviations}
The numbers $\ve_i(\vp):=\dim_k \pi^{i}(F^\vp)$ are called the \emph{deviations of $\vp$}. 
They are encoded in the Poincar\'e series \[P(t)=\sum_{i\geqslant 0}\dim_k (\Tor^{k\lotimes_S R}_i(k,k))t^i\]
of $k$ over the derived fiber $k\lotimes_R S$ according to the equality
\[P(t) = \frac{\prod_{i=1}^\infty(1+t^{2i-1})^{\ve_{2i-1}(\vp)}}{\prod_{i=1}^\infty(1-t^{2i})^{\ve_{2i}(\vp)}}\,.\]  Furthermore, upon fixing a minimal model $R[X]$ for $\vp$, we see that the number of variables adjoined in $X_i$ is exactly $\ve_{i+1}(\vp).$ These naturally generalize the deviations of a local ring; a standard reference for the latter is \cite[Section~7]{Avramov:2010}, and see the references contained in \emph{loc.\@ cit}.
\end{chunk}

\begin{chunk}\label{c:decomposable_deviations}When $\vp$ is a weakly-closed surjective map and $R$ has residual characteristic $p>0$, with acyclic closure $R\langle Y\rangle$, \cref{t:comparison} shows that the deviations of $\vp$ are completely determined by the numbers $\gdev_i(\vp):=\dim_k(\gind_{i-1}(R\langle Y\rangle))$; the latter value is exactly the number of $\Gamma$-variables adjoined in $Y_{i-1}$ and referred to as the $i^{\text{th}}$ $\Gamma$-deviation of $\vp$.
In particular, we have
\[\ve_i(\vp)=
 \begin{dcases}
 \sum_{s=0}^t \gdev_{2jp^s+1}(\vp) & i=2jp^t+1 \\
 \gdev_{2jp^t+2}(\vp)+\sum_{s=0}^{t-1} \gdev_{2jp^s+1}(\vp) & i=2jp^t+2\\
 \gdev_i(\vp) & \text{otherwise}
 \end{dcases}
\]

Furthermore, since the first steps in the inductive constructions of an acyclic closure and a minimal model coincide, the equalities below always hold \[ 
\ve_2(\vp)=\gdev_2(\vp) \ \text{ and } \ \ve_3(\vp)=\gdev_3(\vp)\,.\] 
\end{chunk}
It is known that the eventual vanishing of the deviations of a surjective local map is equivalent to the map being complete intersection when it is a map of finite projective dimension, or more generally, when the map has finite weak category; see \cite[Theorem~C]{Avramov/Halperin:1987} for the former, and \cite[Section~3]{Avramov:1999} for the latter (as well as \cite[Theorem~5.4]{Avramov/Iyengar:2003}). 
The numeric relations listed in \cref{c:decomposable_deviations} establishes the equivalence for a completely different class of surjective homomorphisms:
\begin{corollary}\label{cor:ci}
If $\vp\colon R\rightarrow S$ is a weakly-closed surjective local map with residual characteristic $p>0$, then the following are equivalent: 
\begin{enumerate}
  \item $\vp$ is complete intersection; 
  \item $\ve_i(\vp)=0$ for all $i\gg 0$; 
  \item $\ve_i(\vp)=0$ for $i=2p^t+1$ or $i=2p^t+2$ for some $t\geqslant 1$.
\end{enumerate}
\end{corollary}
Our last main result, another immediate consequence of \cref{c:decomposable_deviations}, shows that rigidity of certain deviations detects the quasi-complete intersection property in positive characteristic among weakly-closed maps; this should be compared with \cite[Theorem~5.3]{Avramov/Henriques/Sega:2013} and \cite[Theorem~33]{Briggs:2018} which characterizes the quasi-complete intersection property in terms of the functorial map $\pi(F^{\rho})\to \pi(F^{\vp\rho})$ where $\rho\colon Q\to R$ is a Cohen presentation of $R$.
\begin{corollary}\label{cor:qci}
Let $\vp\colon R\rightarrow S$ be a weakly-closed surjective local map with residual characteristic $p>0$. Then $\vp$ is a quasi-complete intersection homomorphism if and only if  there are equalities \[\ve_i(\vp)=\begin{cases}
\ve_3(\vp) & \text{ if }i=2p^t+1 \text{ or } i=2p^t+2\\
0   & \text{ otherwise}
\end{cases}\]
 for all $i\geqslant 4$ and $t\geqslant1$. 
\end{corollary}

\begin{remark}
\cref{cor:ci,cor:qci} fail when the residual characteristic is zero.  In particular, if $\vp$ is a quasi-complete intersection that is not a complete intersection, then $\vp$ is a closed map satisfying $\varepsilon_3(\vp)\neq 0$ and $\varepsilon_{i}(\vp)=0$ for $i\geq 4$.  Hence $\vp$ satisfies conditions (2) and (3) of \cref{cor:ci}, while (1) fails.  Furthermore, $\vp$ is a quasi-complete intersection with $\varepsilon_{i}(\vp)=0\neq\varepsilon_{3}(\vp)$ whenever $i=2p^t+1$ or $i=2p^t+2$, so the forward implication of \cref{cor:qci} fails.
\end{remark}

We end this article with some closing remarks. 

\begin{remark}
The results in the present article were stated for surjective homormophisms; they naturally extend---and can be deduced from the surjective case---to the setting of (arbitrary) local homomorphisms using the theory of Cohen factorizations introduced in \cite{Avramov/Foxby/Herzog:1994}. We leave these deductions to the interested reader.
\end{remark}

\begin{remark} \label{r:question}
Recall by \cref{r:decomposability}, $\vp$ is weakly-closed provided for an acyclic closure $R\langle Y\rangle$ for $\vp$, the boundary $\del y$ does not contain a nonzero summand from $kY^{(p^\infty)}$ for each $y\in Y$; this is needed in the core result \cref{t:refinement}.
The examples of such maps listed in \cref{c:closedexamples} are closed, which is a priori stronger than requiring $\vp$ be weakly-closed.
An example of a weakly-closed map that is \emph{not} closed, when $k$ has characteristic $p$, is unknown to the authors. 
\end{remark}

\begin{remark}
The Lie algebra $L$ in \cref{t:abelian} is intrinsic to the map $\vp$. Namely, as a $k$-space $L$ is $
(\shift \gind R\langle Y\rangle)^*$ where $R\langle Y\rangle$ is an acyclic closure of $\vp$. If the characteristic of $k$ is different from 2 or 3, without the assumption $\vp$ is weakly-closed, then $L$ acquires a Lie algebra structure by adjusting the formulas for the bracket and reduced square in \cref{c:HLA}. An investigation of this Lie algebra is reserved for another time as the development of this Lie algebra is not entirely relevant to the content in this work.
\end{remark}

\bibliographystyle{amsplain}
\bibliography{qcibib}

\providecommand{\bysame}{\leavevmode\hbox to3em{\hrulefill}\thinspace}
\providecommand{\MR}{\relax\ifhmode\unskip\space\fi MR }
\providecommand{\MRhref}[2]{%
  \href{http://www.ams.org/mathscinet-getitem?mr=#1}{#2}
}
\providecommand{\href}[2]{#2}
\begin{thebibliography}{10}

\bibitem{Cartan:1956}
\emph{S\'{e}minaire {H}enri {C}artan de l'{E}cole {N}ormale {S}up\'{e}rieure,
  1954/1955. {A}lg\`ebres d'{E}ilenberg-{M}ac{L}ane et homotopie},
  Secr\'{e}tariat Math\'{e}matique, 11 rue Pierre Curie, Paris, 1956, 2\`eme
  \'{e}d.

\bibitem{Avramov:1984}
Luchezar~L. Avramov, \emph{Local algebra and rational homotopy}, Algebraic
  homotopy and local algebra ({L}uminy, 1982), Ast\'{e}risque, vol. 113, Soc.
  Math. France, Paris, 1984, pp.~15--43. \MR{749041}

\bibitem{Avramov:1986}
\bysame, \emph{Golod homomorphisms}, Lecture Notes in Math., vol. 1183,
  Springer, Berlin, 1986. \MR{846439}

\bibitem{Avramov:1999}
\bysame, \emph{Locally complete intersection homomorphisms and a conjecture of
  {Q}uillen on the vanishing of cotangent homology}, Ann. of Math. (2)
  \textbf{150} (1999), no.~2, 455--487. \MR{1726700}

\bibitem{Avramov:2010}
\bysame, \emph{Infinite free resolutions}, Six lectures on commutative algebra,
  Mod. Birkh{\"a}user Class., Birkh{\"a}user Verlag, Basel, 2010, pp.~1--118.
  \MR{2641236}

\bibitem{Avramov/Foxby/Herzog:1994}
Luchezar~L. Avramov, Hans-Bj{\o}rn Foxby, and Bernd Herzog, \emph{Structure of
  local homomorphisms}, J. Algebra \textbf{164} (1994), no.~1, 124--145.
  \MR{1268330}

\bibitem{Avramov/Halperin:1986}
Luchezar~L. Avramov and Stephen Halperin, \emph{Through the looking glass: a
  dictionary between rational homotopy theory and local algebra}, Algebra,
  algebraic topology and their interactions ({S}tockholm, 1983), Lecture Notes
  in Math., vol. 1183, Springer, Berlin, 1986, pp.~1--27. \MR{846435}

\bibitem{Avramov/Halperin:1987}
\bysame, \emph{On the nonvanishing of cotangent cohomology}, Comment. Math.
  Helv. \textbf{62} (1987), no.~2, 169--184. \MR{896094}

\bibitem{Avramov/Henriques/Sega:2013}
Luchezar~L. Avramov, In\^{e}s Bonacho Dos~Anjos Henriques, and Liana~M.
  \c{S}ega, \emph{Quasi-complete intersection homomorphisms}, Pure Appl. Math.
  Q. \textbf{9} (2013), no.~4, 579--612. \MR{3263969}

\bibitem{Avramov/Iyengar:2003}
Luchezar~L. Avramov and Srikanth Iyengar, \emph{Andr\'{e}-{Q}uillen homology of
  algebra retracts}, Ann. Sci. \'{E}cole Norm. Sup. (4) \textbf{36} (2003),
  no.~3, 431--462. \MR{1977825}

\bibitem{Avramov/Iyengar:2005}
\bysame, \emph{Gaps in {H}ochschild cohomology imply smoothness for commutative
  algebras}, Math. Res. Lett. \textbf{12} (2005), no.~5-6, 789--804.
  \MR{2189239}

\bibitem{Avramov/Iyengar:2010}
Luchezar~L. Avramov and Srikanth~B. Iyengar, \emph{Cohomology over complete
  intersections via exterior algebras}, Triangulated categories, London Math.
  Soc. Lecture Note Ser., vol. 375, Cambridge Univ. Press, Cambridge, 2010,
  pp.~52--75. \MR{2681707}

\bibitem{Bergh/Celikbas/Jorgensen:2014}
Petter~Andreas Bergh, Olgur Celikbas, and David~A. Jorgensen, \emph{Homological
  algebra modulo exact zero-divisors}, Kyoto J. Math. \textbf{54} (2014),
  no.~4, 879--895. \MR{3276421}

\bibitem{Henriques/Sega:2011}
In\^{e}s Bonacho Dos Anjos~Henriques and Liana~M. \c{S}ega, \emph{Free
  resolutions over short {G}orenstein local rings}, Math. Z. \textbf{267}
  (2011), no.~3-4, 645--663. \MR{2776052}

\bibitem{Briggs:2018}
Benjamin Briggs, \emph{Local commutative algebra and hochschild cohomology
  through the lens of koszul duality}, Ph.D. thesis, University of Toronto
  (2018), \url{https://www.math.utah.edu/~briggs/briggsthesis.pdf}.

\bibitem{Briggs:2020}
\bysame, \emph{Vasconcelos' conjecture on the conormal module}, Invent. Math.
  \textbf{227} (2022), no.~1, 415--428. \MR{4359479}

\bibitem{Gulliksen:1968}
Tor~H. Gulliksen, \emph{A proof of the existence of minimal {$R$}-algebra
  resolutions}, Acta Math. \textbf{120} (1968), 53--58. \MR{224607}

\bibitem{Gulliksen:1974}
\bysame, \emph{A change of ring theorem with applications to {P}oincar{\'e}
  series and intersection multiplicity}, Math. Scand. \textbf{34} (1974),
  167--183. \MR{364232}

\bibitem{Gulliksen/Levin:1969}
Tor~H. Gulliksen and Gerson Levin, \emph{Homology of local rings}, Queen's
  Paper in Pure and Applied Mathematics, No. 20, Queen's University, Kingston,
  Ont., 1969. \MR{0262227}

\bibitem{Halperin/1987}
Stephen Halperin, \emph{The nonvanishing of the deviations of a local ring},
  Comment. Math. Helv. \textbf{62} (1987), no.~4, 646--653. \MR{920063}

\bibitem{Hardy/Wright:2008}
Godfrey~H. Hardy and Edward~M. Wright, \emph{An introduction to the theory of
  numbers}, sixth ed., Oxford University Press, Oxford, 2008, Revised by D. R.
  Heath-Brown and J. H. Silverman, With a foreword by Andrew Wiles.
  \MR{2445243}

\bibitem{Kustin/Sega:2018}
Andrew~R Kustin and Liana~M Sega, \emph{The structure of quasi-complete
  intersection ideals}, arXiv preprint arXiv:1809.11094 (2018),
  \url{https://arxiv.org/abs/1809.11094}.

\bibitem{Levin:1980}
Gerson Levin, \emph{Large homomorphisms of local rings}, Math. Scand.
  \textbf{46} (1980), no.~2, 209--215. \MR{591601}

\bibitem{Lucas:1878}
Edouard Lucas, \emph{Th\'{e}orie des fonctions num\'{e}riques simplement
  p\'{e}riodiques}, American Journal of Mathematics \textbf{1} (1878), no.~2,
  184–196.

\bibitem{Lutz:2016}
Jason~M. Lutz, \emph{Homological characterizations of quasi-complete
  intersections}, ProQuest LLC, Ann Arbor, MI, 2016, Thesis (Ph.D.)--The
  University of Nebraska - Lincoln. \MR{3542277}

\bibitem{Majadas/Rodicio:2010}
Javier Majadas and Antonio~G. Rodicio, \emph{Smoothness, regularity and
  complete intersection}, London Mathematical Society Lecture Note Series, vol.
  373, Cambridge University Press, Cambridge, 2010. \MR{2640631}

\bibitem{Quillen:1970}
Daniel Quillen, \emph{On the (co-) homology of commutative rings}, Applications
  of {C}ategorical {A}lgebra ({P}roc. {S}ympos. {P}ure {M}ath., {V}ol. {XVII},
  {N}ew {Y}ork, 1968), Amer. Math. Soc., Providence, R.I., 1970, pp.~65--87.
  \MR{0257068}

\bibitem{Roby:1980}
Norbert Roby, \emph{Lois polyn\^{o}mes multiplicatives universelles}, C. R.
  Acad. Sci. Paris S\'{e}r. A-B \textbf{290} (1980), no.~19, A869--A871.
  \MR{580160}

\bibitem{Rodicio:1992}
Antonio~G. Rodicio, \emph{On the free character of the first {K}oszul homology
  module}, J. Pure Appl. Algebra \textbf{80} (1992), no.~1, 59--64.
  \MR{1167387}

\bibitem{Schoeller:1967}
Colette Schoeller, \emph{Homologie des anneaux locaux noeth\'{e}riens}, C. R.
  Acad. Sci. Paris S\'{e}r. A-B \textbf{265} (1967), A768--A771. \MR{224682}

\bibitem{Tate:1957}
John Tate, \emph{Homology of {N}oetherian rings and local rings}, Illinois J.
  Math. \textbf{1} (1957), 14--27. \MR{86072}

\bibitem{Vasconcelos:1978}
W.~V. Vasconcelos, \emph{On the homology of {$I/I^{2}$}}, Comm. Algebra
  \textbf{6} (1978), no.~17, 1801--1809. \MR{508082}

\bibitem{Windle:2017}
Andrew Windle, \emph{Cohomological {O}perators on {Q}uotients by
  {Q}uasi-{C}omplete {I}ntersection {I}deals},  (2017), 74, Thesis (Ph.D.)--The
  University of Nebraska - Lincoln. \MR{3732050}

\end{thebibliography}

\end{document}